\documentclass{amsart}

\usepackage{amsfonts,amssymb,amsmath}

\theoremstyle{plain}
\newtheorem{thm}{Theorem}[section]
\newtheorem{lemma}[thm]{Lemma}
\newtheorem{prop}[thm]{Proposition}
\newtheorem{cor}[thm]{Corollary}

\theoremstyle{remark}

\theoremstyle{definition}
\newtheorem{defn}[thm]{Definition}
\newtheorem{example}[thm]{Example}

\newcommand{\Z}{\mathbb{Z}}
\newcommand{\Q}{\mathbb{Q}}

\newcommand{\Spec}{\operatorname{Spec}}

\newcommand{\Ext}{\operatorname{Ext}}

\renewcommand{\mod}{\operatorname{mod}}
\newcommand{\Mod}{\operatorname{Mod}}
\newcommand{\End}{\operatorname{End}}
\newcommand{\Id}{\operatorname{Id}}

\newcommand{\Vect}{{\sf{Vect}}}
\renewcommand{\a}{\alpha}
\renewcommand{\b}{\beta}

\newcommand{\Aut}{\operatorname{Aut}}
\newcommand{\Emb}{\operatorname{Emb}}
\renewcommand{\span}{\operatorname{span}}
\newcommand{\im}{\operatorname{im}}
\newcommand{\diag}{\operatorname{diag}}

\begin{document}
\title{Two-sided vector spaces}
\author{Adam Nyman}
\address{Department of Mathematics, University of Montana, Missoula, MT 59812-0864}
\email{NymanA@mso.umt.edu}
\author{Christopher J. Pappacena}
\address{Department of Mathematics, Baylor University, Waco, TX 76798}
\email{Chris\_$\,$Pappacena@baylor.edu}
\keywords{Two-sided vector space, matrix homomorphism, noncommutative vector bundle}
\date{\today}
\thanks{2000 {\it Mathematics Subject Classification. } Primary  15A03, 16D20; Secondary 14A22, 19A49, 19D50}
\thanks{The second author was partially supported by the National Security Agency under grant NSA 032-5529.}

\begin{abstract}
We study the structure of two-sided vector spaces over a perfect field $K$.  In particular, we give a complete characterization of isomorphism classes of simple two-sided vector spaces which are left finite-dimensional. Using this description, we compute the Quillen $K$-theory of the category of left finite-dimensional, two-sided vector spaces over $K$.  We also consider the closely related problem of describing homomorphisms $\phi:K\rightarrow M_n(K)$.
\end{abstract}

\maketitle

\section{Introduction}

Given the central role that vector spaces play in mathematics, it is natural to study two-sided vector spaces; that is, abelian groups $V$ equipped with both a left and right action by a field $K$, subject to the associativity condition $(xv)y=x(vy)$ for $x,y\in K$ and $v\in V$.  When the left and right actions of $K$ on $V$ agree, then $V$ is nothing more than an ordinary $K$-vector space.  In this case, $V$ decomposes into a direct sum of irreducible subspaces, and every irreducible subspace is $1$-dimensional (and hence isomorphic to $K$ as a vector space over $K$).  When the left and right actions of $K$ and $V$ differ, then the structure of $V$ can be much more complicated.  For example, $V$ does not generally decompose into irreducible subspaces.  Furthermore, the distinct irreducible subspaces of $V$ may not be $1$-dimensional or isomorphic to each other.  

Apart from being intrinsically interesting, two-sided vector spaces play an important role in noncommutative algebraic geometry.  In particular, two-sided vector spaces are noncommutative analogues of vector bundles over $\operatorname{Spec }K$.  Noncommutative analogues of vector bundles were defined and used by Van den Bergh \cite{vdb1} to construct noncommutative $\mathbb{P}^{1}$-bundles over commutative schemes.   

The purpose of this paper is to study the structure of two-sided vector spaces over $K$ when $K$ is a perfect field.  In particular, we classify irreducible two-sided vector spaces which are finite-dimensional as ordinary $K$-vector spaces.  We then use our classification to determine the algebraic $K$-theory of the category of all such two-sided vector spaces.  We also give canonical representations for certain two-sided vector spaces, generalizing \cite[Theorem 1.3]{pat1}.  

The structure theory of two-sided vector spaces has important applications to noncommutative algebraic geometry via the theory of noncommutative vector bundles.  Let $S$ and $X$ be commutative schemes and suppose $X$ is an $S$-scheme of finite type.  By an ``$S$-central noncommutative vector bundle over $X$" we mean an $\mathcal{O}_{S}$-central, coherent sheaf $X$-bimodule which is locally free on the right and left \cite[Definition 2.3, p. 440]{vdb1}.  When $S=\Spec k$ and $X=\Spec K$, a sheaf $X$-bimodule which is locally free of finite rank on each side is nothing more than a two-sided $K$-vector space $V$, finite-dimensional on each side, where the left and right actions of $K$ on $V$ may differ.

When $X$ is an integral scheme, any noncommutative vector bundle $\mathcal{E}$ over $X$ localizes to a noncommutative vector bundle ${\mathcal{E}}_{\eta}$ over the generic point $\eta$ of $X$.  If ${\mathcal{O}}_{X}$ acts centrally on $\mathcal{E}$, then ${\mathcal{E}}_{\eta}$ is completely characterized by its dimension over the field of fractions, $k(X)$, of $X$.  In this case, the rank of $\mathcal{E}$ is defined as $\operatorname{dim}_{k(X)}{\mathcal{E}}_{\eta}$.  Since localization is exact, localization induces a map $K_{0}(X) \rightarrow K_{0}(\operatorname{Spec }k(X))$, and the rank of $\mathcal{E}$ can also be defined as the image of the class of $\mathcal{E}$ via this map.

Now suppose $X$ is of finite type over $\Spec k$.  If ${\mathcal{O}}_{X}$ does not act centrally on $\mathcal{E}$, then ${\mathcal{E}}_{\eta}$ will be a two-sided vector space over $k(X)$ whose left and right actions differ.  In this case, ${\mathcal{E}}_{\eta}$ is not completely characterized by its left and right dimension.  However, localization induces a map $K_0^B(X) \rightarrow K_0^B(\Spec k(X))$ where $K_{0}^B(X)$ denotes the Quillen $K$-theory of the category of $k$-central noncommutative vector bundles over $X$ and $K_0^B(\Spec k(X))$ is defined similarly.  It is thus reasonable to define the rank of $\mathcal{E}$ as the image of the class of $\mathcal{E}$ via this map.  If this notion of rank is to be useful we must be able to compute the group $K^B_0(\Spec k(X))$.    

In addition, one can often construct a noncommutative symmetric algebra $\mathcal{A}$ from a noncommutative vector bundle $\mathcal{E}$ \cite[Section 2]{pat1}, \cite[Section 5.1]{vdb2}.  While $\mathcal{A}$ is not generally a sheaf of algebras over $X$, its localization at the generic point $\eta$ of $X$, ${\mathcal{A}}_{\eta}$, is an algebra.  The birational class of the projective bundle associated to $\mathcal{A}$ is determined by the degree zero component of the skew field of fractions of ${\mathcal{A}}_{\eta}$.  Since ${\mathcal{A}}_{\eta}$ is generated by ${\mathcal{E}}_{\eta}$, we see that the birational class of a noncommutative projectivization is governed by a noncommutative vector bundle over $\Spec K(X)$.  

We now summarize the contents of the paper.  In Section 2 we describe some general properties of two-sided vector spaces that we will use in the sequel.  In Section 3 we study simple objects in $\Vect(K)$, the category of two-sided $K$-vector spaces which are left finite-dimensional.  In particular, we parameterize isomorphism classes of simple two-sided vector spaces by orbits of embeddings $\lambda:K \rightarrow \bar{K}$ under the action of left-composition by elements of $\Aut(\bar K/K)$ (Theorem \ref{simple theorem}).  In Section 4, we use results from Section 3 to explicitly describe the Quillen $K$-groups of $\Vect(K)$, denoted $K^B_{i}(K)$ (Theorem \ref{kgroup theorem}), and give a procedure for calculating the ring structure on $K_0^B(K)$. 
    
Finally in Section 5, we study matrix representations of two-sided vector spaces, i.e. homomorphisms $\phi:K \rightarrow M_{n}(K)$.  Specifically, we consider the problem of finding a $P\in GL_n(K)$ such that the homomorphism $P\phi P^{-1}$ has a particularly nice form.  We prove that if every matrix in $\im\phi$ has all of its eigenvalues in $K$, then the triangularized form of $\phi$ can be described in terms of higher derivations on $K$ (Theorem \ref{upper triangular thm}).  We also develop sufficient conditions on a matrix $A$ to ensure the existence of an upper triangular matrix $P\in GL_n(K)$ with $PAP^{-1}$ in Jordan canonical form (Theorem \ref{theorem.jcf}).  Combining these results, we give sufficient conditions that enable us to describe the off diagonal blocks of $P\phi P^{-1}$ (Corollary \ref{jcf cor}).

Throughout the paper, we provide examples of our results.  We reproduce and extend the third case of \cite[Theorem 1.3]{pat1} by describing the structure of $2$ and $3$-dimensional simple two-sided vector spaces when they exist.  When $p\geq 3$ is prime and $K = \Q(\sqrt[p]{2})$, we describe the isomorphism classes of $\Q$-central two-sided $K$-vector spaces.  There are only two, with dimensions $1$ and $p-1$.  We then describe the ring $K_0^B(K)$ via generators and relations.  Finally, we provide an example in Section 5 to show that there exists a field $K$, a homomorphism $\phi:K\rightarrow M_3(K)$, and an element $y\in K$ such that there is no $P\in GL_3(K)$ with $P\phi P^{-1}$ upper triangular and $P\phi(y)P^{-1}$ in Jordan canonical form (Example \ref{example.fail}).  

\subsection*{Acknowledgments} We thank R. Piziak for general help with some of the finer points of linear algebra and matrix theory, and we thank R. Guralnick for informing us of a more general version of Lemma \ref{Piziak lemma} than that which appeared in earlier drafts of this paper.

\section{Preliminaries}
As we mentioned above, $K$ will always denote a perfect field of arbitrary characteristic and $\bar K$ will be a fixed algebraic closure of $K$.  By a \emph{two-sided vector space} we mean a $K$-bimodule $V$ where the left and right actions of $K$ on $V$ do not necessarily coincide.  Except when explicitly stated to the contrary, we shall only consider those two-sided vector spaces whose left dimension is finite, and we use the phrases ``two-sided vector space" and ``bimodule" interchangeably.

Since we shall only consider bimodules $V$ with $_KV$ and $V_K$ both unital, it is easy to see that the prime subfield of $K$ must act centrally on any two-sided vector space.  We shall fix a base field $k\subset K$ and consider only those bimodules $V$ which are centralized by $k$.  Note that we do not assume that $K/k$ is algebraic in general.  While all of the notions that we introduce in this paper will depend on the centralizing subfield $k$, it turns out that $k$ itself will usually not play an important role in any of our results.  In particular we will omit $k$ from our notation.  

Given a $K$-bimodule $V$ and a set of vectors $\{v_i:i\in I\}$, we shall always write $\span\{v_i\}$ to stand for the \emph{left} span of the $v_i$.  In general, $\span\{v_i\}$ will not be a sub-bimodule of $V$.

If $V$ is a two-sided vector space, then right multiplication by $x\in K$ defines an endomorphism $\phi(x)$ of $_KV$, and the right action of $K$ on $V$ is via the $k$-algebra homomorphism $\phi:K\rightarrow \End(_KV)$.  This observation motivates the following definition.

\begin{defn} Let $\phi:K\rightarrow M_n(K)$ be a nonzero homomorphism.  Then we denote by $_1K^n_\phi$ the two-sided vector space of left dimension $n$, where the left action is the usual one and the right action is via $\phi$; that is, 
\begin{equation}x\cdot(v_1,\dots, v_n)=(xv_1,\dots,xv_n),\ \ \ (v_1,\dots, v_n)\cdot x=(v_1,\dots,v_n)\phi(x).\end{equation}
We shall always write scalars as acting to the left of elements of $_1K^n_\phi$ and matrices acting to the right; thus, elements of $K^n$ are written as row vectors and if $v\in K^n$ is an eigenvector for $\phi(x)$ with eigenvalue $\lambda$, we write $v\phi(x)=\lambda v$.
\end{defn}

It is easy to see that, if $V$ is a two-sided vector space and $[K:k]<\infty$, then $\dim{_KV}$ is finite if and only if $\dim V_K$ is finite, and in this case the two dimensions must be equal. Thus, when $[K:k]< \infty$, we may drop subscripts and simply write $\dim V$ for this common dimension.  If $[K:k]$ is infinite, it is no longer true that the finiteness of $\dim{_KV}$ implies the finiteness of $\dim{V_K}$, as the following example shows.

\begin{example} \label{infinite-dimensional example}
Let $K=k(x_{1},x_{2}, \ldots )$, let $\phi:K \rightarrow K$ be the homomorphism defined by $\phi(x_{i})=x_{i+1}$ and let $V = {_1K_\phi}$.  Then the dimension of $_KV$ is $1$, while the dimension of $V_K$ is infinite.
\end{example}

We denote the category of left finite-dimensional two-sided vector spaces by $\Vect(K)$.  Clearly $\Vect(K)$ is a finite-length category. If we write $K^e=K\otimes_kK$ for the enveloping algebra of $K$, then there is a category equivalence between (not necessarily finite-dimensional) $K$-bimodules and (say) left $K^e$-modules.  Under this equivalence, $\Vect(K)$ can be identified as a full subcategory of the category of finite-length $K^e$-modules.  If $[K:k]$ is finite, then $\Vect(K)=K^e\text{-}\mod$, the category of noetherian left $K^e$-modules.  When $K/k$ is infinite, this need no longer hold:  if we define $V={_\phi K_1}$ in the obvious way for the map $\phi$ in Example \ref{infinite-dimensional example}, then $V$ is clearly simple in $K^e\text{-}\Mod$ but is not in $\Vect(K)$.

If $V\in\Vect(K)$ with left dimension equal to $n$, then choosing a left basis for $V$ shows that $V\cong {_1K^n_\phi}$ for some homomorphism $\phi:K\rightarrow M_n(K)$; we shall say that $\phi$ \emph{represents} $V$ in this case.

If $L$ is an extension field of $K$, then of course any matrix over $K$ can be viewed as a matrix over $L$, and a function $\phi:K\rightarrow M_n(K)$ can be viewed as having its image in $M_n(L)$.  If $A,B\in M_n(K)$, then we write $A\sim_LB$ if $A$ and $B$ are similar in $M_n(L)$; that is, if $B=PAP^{-1}$ for some $P\in GL_n(L)$.  Similarly, if $\phi:K\rightarrow M_n(K)$ and $\psi:K\rightarrow M_n(K)$ are functions, we write $\phi\sim_L\psi$ if $\phi(x)=P\psi(x)P^{-1}$ for some $P\in GL_n(L)$.  In either case, if $P$ actually lives in $M_n(K)$, then we simply write $\sim$ for $\sim_K$.  

The following well known result follows readily from the fact that a homomorphism $\phi:K \rightarrow M_n(K)$ restricts to a representation of the group $K^{*}$ of units of $K$.

\begin{lemma} Let $L$ be an extension field of $K$.  $L \otimes_{K} {_1K^n_\phi}\cong{L \otimes_{K} {_1K^n_\psi}}$ as $L \otimes_{K} K^{e}$-modules if and only if $\phi\sim_L\psi$.\label{isomorphism lemma}
\end{lemma}

The next result is a special case of the Noether-Deuring Theorem \cite[Exercise 6, p. 139]{cr}.  

\begin{lemma} Let $L$ be an extension field of $K$, and let $A,B\in M_n(K)$.  If $A\sim_LB$, then $A\sim B$.  Similarly, if $\phi:K\rightarrow M_n(K)$ and $\psi:K\rightarrow M_n(K)$ are functions with $\phi\sim_L\psi$, then $\phi\sim \psi$.
\label{Piziak lemma}
\end{lemma}

\section{Simple two-sided vector spaces}
The main result of this section is a determination of all of the isomorphism classes of simple two-sided vector spaces. In order to state our classification, we introduce some notation.  We write $\Emb(K)$ for the set of $k$-embeddings of $K$ into $\bar K$, and $G=G(K)$ for the absolute Galois group $\Aut(\bar K/K)$.  (Note that $\bar K/K$ is Galois since $K$ is perfect.)  If $L$ is an intermediate field, then we write $G(L)$ for $\Aut(\bar K/L)$.  

Now, $G$ acts on $\Emb(K)$ by left composition. Given $\lambda\in \Emb(K)$, we denote the orbit of $\lambda$ under this action by $\lambda^G$, and we write $K(\lambda)$ for the composite field $K\vee\im(\lambda)$.  The stabilizer $G_\lambda$ of $\lambda$ under this action is easy to calculate:  $\sigma\lambda=\lambda$ if and only if $\sigma$ fixes $\im(\lambda)$; since $\sigma$ fixes $K$ as well we have that $G_\lambda=G(K(\lambda))$.

\begin{lemma} $[K(\lambda):K]$ is finite if and only if $|\lambda^G|$ is finite, and in this case $|\lambda^G|=[K(\lambda):K]$. \label{finite orbit lemma}
\end{lemma}

\begin{proof} By the above, the stabilizer of $\lambda$ is $G(K(\lambda))$.  Thus $|\lambda^G|=[G:G(K(\lambda))]$.  The result now follows by basic Galois Theory. 
\end{proof}

It turns out that we will only be interested in those embeddings $\lambda$ with $\lambda^G$ finite; we denote the set of finite orbits of $\Emb(K)$ under the action of $G$ by $\Lambda(K)$.  The following theorem gives our classification of simple bimodules.

\begin{thm} \label{theorem.main} There is a one-to-one correspondence between isomorphism classes of simples in $\Vect(K)$ and $\Lambda(K)$. Moreover, if $V$ is a simple two-sided vector space corresponding to $\lambda^G\in \Lambda(K)$, then $\dim _KV =|\lambda^G|$ and $\End(V)\cong K(\lambda)$.  \label{simple theorem}
\end{thm}

To prove the first part of Theorem \ref{simple theorem}, we construct a map from the collection of simple bimodules to $\Lambda(K)$ and show that it gives the desired bijection.  We begin in greater generality, starting with a (not necessarily simple) two-sided vector space $V$ with $V\cong{_1K^n_\phi}$.  Now, $\im\phi$ is a set of pairwise commuting matrices in $M_n(K)$; viewing $\im\phi$ as a subset of $M_n(\bar K)$, we know that there exists a common eigenvector $v\in\bar K^n$ for $\im\phi$.  Define a function $\lambda:K\rightarrow \bar K$ by letting $\lambda(x)$ be the eigenvalue of $\phi(x)$ corresponding to $v$; i.e. $v\phi(x)=\lambda(x)v$.  It is easy to check that $\lambda$ is an embedding of $K$ into $\bar K$, and since $\phi$ is a $k$-algebra homomorphism we have that $\lambda\in\Emb(K)$.     

\begin{lemma} \label{lemma.lambdabound} If $v\in \bar K^n$ is a common eigenvector for $\im\phi$ with corresponding eigenvalue $\lambda$, then $\lambda\in\Lambda(K)$.  Moreover, $|\lambda^G|\leq n$.  
\end{lemma}

\begin{proof} Note first that if $\sigma\in G$, $\sigma(v)$ is also a common eigenvector of $\im\phi$, with corresponding eigenvalue $\sigma\lambda$.  Indeed, we compute
\begin{equation}\sigma(v)\phi(x)=\sigma(v)\sigma(\phi(x))=\sigma(v\phi(x))=\sigma(\lambda(x)v)=\sigma\lambda(x)\sigma(v).\end{equation}

Now, if $\sigma\lambda\neq\tau\lambda$, then for at least one value of $x\in K$ the vectors $\sigma(v)$ and $\tau(v)$ are eigenvectors for $\phi(x)$ with different eigenvalues; from this it follows that $\sigma(v)$ and $\tau(v)$ are linearly independent.  If $\lambda^G=\{\sigma_i\lambda:i\in I\}$, then $\{\sigma_i(v):i\in I\}$ is a linearly independent subset of $\bar K^n$.  Thus $|\lambda^G|\leq n$ and in particular $\lambda^G\in\Lambda(K)$.
\end{proof}

Viewing $\lambda$ as an embedding of $K$ into $K(\lambda)$, we may without loss of generality assume that the common eigenvector $v$ for $\im\phi$ with eigenvalue $\lambda$ lives in $K(\lambda)^n$.  We now fix notation which will be useful when proving Theorem \ref{theorem.main}.  We let $m=[K(\lambda):K]=|\lambda^G|$ and we fix a basis $\{\a_1,\dots, \a_m\}$ for $K(\lambda)/K$.  We may write 
\begin{equation} \label{eqn.v}
v=\sum_{i=1}^m{\a_iv_i}
\end{equation}
with each $v_i\in K^n$ and 
\[\lambda(x)=\sum_{i=1}^m\lambda_i(x)\a_i\]
where each $\lambda_i:K\rightarrow K$ is an additive function.  Finally, we let $\b_{ijk}$ denote the structure constants for the basis $\{\a_1,\dots, \a_m\}$; that is, 
\[\a_i\a_j=\sum_{k=1}^m\b_{ijk}\a_k.\]

\begin{lemma} \label{lemma.lambdaexact}  In the above notation, $\span\{v_1,\dots,v_m\}$ is a two-sided subspace of $V$.  In particular, if $V$ is simple, $\dim _KV=|\lambda^{G}|$. \end{lemma}

\begin{proof} We must show that $v_i\phi(x)\in\span\{v_1,\dots, v_m\}$ for all $x\in K$ and all $i$.  On the one hand, $v\phi(x)=\bigl(\sum_i\a_iv_i\bigr)\phi(x)=\sum_i\a_iv_i\phi(x)$.  On the other hand, 
\begin{equation}\begin{split}v\phi(x)=\lambda(x)v&=\bigl(\sum_p\lambda_p(x)\a_p\bigr)\bigl(\sum_q\a_qv_q\bigr)\\
&=\sum_{p,q}\lambda_p(x)\a_p\a_qv_q=\sum_i\a_i\bigl(\sum_{p,q}\beta_{pqi}\lambda_p(x)v_q\bigr).\end{split}\end{equation}
Matching up coefficients of $\a_i$ shows that $v_i\phi(x)=\sum_{p,q}\beta_{pqi}\lambda_p(x)v_q$, so that $v_i\phi(x)\in\span\{v_1,\dots,v_m\}$.  This proves the first assertion.

If $V$ is simple, the first part of the lemma implies $V=\span\{v_1,\dots,v_m\}$.  Thus, $m = |\lambda^G| \geq \dim _KV$.  On the other hand, $|\lambda^{G}| \leq \dim _KV$ by Lemma \ref{lemma.lambdabound}.  Thus, $|\lambda^{G}|=\dim _KV$ when $V$ is simple.    
\end{proof}

\begin{prop} \label{prop.bij}
Let $\phi:K \rightarrow M_{n}(K)$ be a homomorphism and let $\lambda:K \rightarrow \overline{K}$ be the eigenvalue of a common eigenvector of $\operatorname{im }\phi \subset M_{n}(\overline{K})$.  The map 
$$
\Phi:\{\mbox{Isomorphism classes of simples in }\Vect(K) \} \rightarrow \Lambda(K)
$$
defined by $\Phi([_1K^{n}_\phi])=\lambda^{G}$ is a bijection. 
\end{prop}

\begin{proof} {\it Part 1.}  We show $\Phi$ is an injection.
\newline
\newline
{\it Part 1, Step 1.  We show $\Phi$ is well defined.}  Let $V$ be a simple object in $\Vect(K)$, and suppose $V \cong {_1K^n_\phi}$.  By Lemma \ref{lemma.lambdaexact}, $|\lambda^G|=n$.  Let us write out the elements of $\lambda^G$ as $\{\lambda,\sigma_2\lambda,\dots,\sigma_n\lambda\}$.  Then taking $\{v,\sigma_2(v),\dots, \sigma_n(v)\}$ as a basis for $\bar K^n$, we see that there exists $Q\in GL_n(\bar K)$ such that 
\begin{equation}Q\phi(x)Q^{-1}=\diag(\lambda(x),\sigma_2\lambda(x),\dots,\sigma_n\lambda(x))\label{similarity equation}\end{equation}
for all $x\in K$.  In particular, if $\mu:K\rightarrow \bar K$ is the eigenvalue for $\phi(x)$ corresponding to some common eigenvector $w$ of $\im\phi$, then we must have $\mu=\sigma_i\lambda$ for some $i$; that is, $\mu^G=\lambda^G$.  

If we choose a different isomorphism $V\cong {_1K^n_\psi}$, then $\phi\sim\psi$; say  $\phi\cong P\psi P^{-1}$ for some $P\in GL_n(K)$.  If $v$ is a common eigenvector for $\im\phi$ with corresponding eigenvalue $\lambda$, then an easy computation shows that $vP$ is a common eigenvector for $\im(\psi)$ with corresponding eigenvalue $\lambda$.  
\newline
\newline
{\it Part 1, Step 2.  We show $\Phi$ is an injection.}  If $K$ is finite, then every embedding of $K$ into $\bar K$ is in fact an automorphism of $K$.  Hence every simple in $\Vect(K)$ is isomorphic to $_1K_\phi$ for some $\phi\in\Aut(K)$, and the above correspondence just sends $_1K_\phi$ to $\phi$.  Thus the claim follows when $K$ is finite.  

Now suppose that $K$ is infinite, $\Phi([V])=\lambda^{G}=\Phi([W])$ and $|\lambda^G|=n$.  Write $V\cong{_1K^n_\phi}$ and $W\cong{_1K^n_\psi}$.  As in equation \eqref{similarity equation}, there are invertible matrices $P,Q\in M_n(\bar K)$ such that 
\begin{equation}P\phi(x)P^{-1}=Q\psi(x)Q^{-1}=\diag(\lambda(x),\sigma_2\lambda(x),\dots,\sigma_n\lambda(x)),
\end{equation}
so that $\phi\sim_{\bar K} \psi$.  By Lemma \ref{Piziak lemma}, $\phi\sim\psi$ and $V\cong W$.
\newline
\newline
{\it Part 2.}  Let $\lambda:K\rightarrow \bar K$ be an embedding with $\lambda^G\in \Lambda(K)$.  We shall construct a simple two-sided vector space $V(\lambda)={_1K^n_{\phi}}$ from $\lambda$, such that $v=(\a_1,\dots,\a_n)\in K(\lambda)^n$ is a common eigenvector for $\im\phi$, with corresponding eigenvalue $\lambda$. Retaining the above notation, we define a map $\phi=(\phi_{ij}):K\rightarrow M_n(K)$ by 
\begin{equation}\phi_{ij}(x)=\sum_{k=1}^n\beta_{jki}\lambda_k(x).\label{phi equation}\end{equation}
\newline
{\it Part 2, Step 1.  We prove that, for all $\sigma\in G$ and $x\in K$, $\sigma(v)$ is an eigenvector for $\phi(x)$ with eigenvalue $\sigma\lambda(x)$.}  We have $\sigma(v)=(\sigma(\a_1),\dots,\sigma(\a_n))$ and $\sigma\lambda(x)=\sum_{i=1}^n\lambda_i(x)\sigma(\a_i)$.  On the one hand,
\begin{equation}\begin{split}
\sigma(v)\phi(x)&=(\sigma(\a_1),\dots,\sigma(\a_n))\phi(x)\\
&=\bigl(\sum_i\phi_{i1}(x)\sigma(\a_i),\dots,\sum_i\phi_{in}(x)\sigma(\a_i)\bigr)\\
&=\bigl(\sum_{i,k}\beta_{1ki}\lambda_k(x)\sigma(\a_i),\dots, \sum_{i,k}\beta_{nki}\lambda_k(x)\sigma(\a_i)\bigr).
\end{split}\end{equation}
On the other hand, 
\begin{equation}
\begin{split}
\sigma\lambda(x)\sigma(v)&=\bigl(\sum_{k}\lambda_k(x)\sigma(\a_k)\bigr)(\sigma(\a_1),\dots,\sigma(\a_n))\\
&=\bigl(\sum_k\lambda_k(x)\sigma(\a_k\a_1),\dots,\sum_k\lambda_k(x)\sigma(\a_k\a_n)\bigr)\\
&=\bigl(\sum_{i,k}\lambda_k(x)\beta_{k1i}\sigma(\a_i),\dots,\sum_{i,k}\lambda_k(x)\beta_{kni}\sigma(\a_i)\bigr)
\end{split}
\end{equation}
Comparing coordinates and using the identity $\beta_{pqr}=\beta_{qpr}$ for all $p,q,r$ gives the result.
\newline
\newline
{\it Part 2, Step 2.  We show $\phi$ is a homomorphism.}  Since each $\lambda_k$ is an additive function it is clear that $\phi$ is additive.  To see that $\phi$ is multiplicative, write out $\lambda^G=\{\sigma_1\lambda,\dots,\sigma_n\lambda\}$ (where $\sigma_1$ is the identity).  Then $\{\sigma_1(v),\dots,\sigma_n(v)\}$ is a basis for $\bar K^n$, and for all $x,y\in K$, we have
\begin{multline}
\sigma_i(v)\phi(x)\phi(y)=\sigma_i\lambda(x)\sigma_i(v)\phi(y)=\sigma_i\lambda(x)\sigma_i\lambda(y)\sigma_i(v)\\=\sigma_i\lambda(xy)\sigma_i(v)=\sigma_i(v)\phi(xy).\end{multline}
This shows that $\phi(x)\phi(y)$ and $\phi(xy)$ act as the same linear transformation on each $\sigma_i(v)$.  Since the $\sigma_i(v)$ form a basis for $\bar K^n$, we have that $\phi(x)\phi(y)=\phi(xy)$ for all $x,y\in K$.
\newline
\newline
{\it Part 2, Step 3.  Since $\phi$ is a homomorphism, we can define the two-sided vector space $V(\lambda)={_1K^n_\phi}$.  We prove $V(\lambda)$ is simple.}  Suppose that $W$ is a simple sub-bimodule of $V(\lambda)$ with $\dim W=m$, and fix a left basis for $V(\lambda)$ containing a left basis for $W$.  Then, relative to this basis, we have $V(\lambda)\cong{_1K^n_\psi}$, where $\psi=\begin{pmatrix}\psi_1&\theta\\0&\psi_2\end{pmatrix}$ and $W\cong{_1K^m_{\psi_2}}$.  Since $W$ is simple, there is a unique orbit $\mu^G=\{\mu_1,\dots,\mu_m\}\in\Lambda(K)$ with $\psi_2\sim_{\bar K}\diag(\mu_1,\dots,\mu_m)$.  On the other hand, we have by definition of $V(\lambda)$ that $\phi\sim_{\bar K}\diag(\lambda,\sigma_2\lambda,\dots,\sigma_n\lambda)$; since $\phi\sim\psi$ we see that $\mu_1=\sigma_j\lambda$ for some $j$.  Hence $\mu^G=\lambda^G$ and $W=V(\lambda)$ since $\Phi$ is injective.
\end{proof}

To complete the proof of Theorem \ref{simple theorem}, we need to compute $\End(V(\lambda))$.  

\begin{prop} $\End(V(\lambda))\cong K(\lambda)$.\end{prop}

\begin{proof}  Let $|\lambda^{G}|=n$.  We first note that $\End(V(\lambda))$ can be made into a left vector space over $K$ by defining $(xf)(v)=xf(v)$ for $x\in K$, $v\in V(\lambda)$, and $f\in\End(V(\lambda))$.  Also, since $V(\lambda)$ is a simple bimodule, it is generated as a bimodule by a single element $w$.  If $\{f_1,\dots, f_{n+1}\}$ is a subset of $\End(V(\lambda))$, then $\{f_1(w),\dots, f_{n+1}(w)\}$ are necessarily linearly dependent in $V(\lambda)$; hence there exist $x_i\in K$ such that the endomorphism $\sum_{i=1}^{n+1}x_if_i$ acts as $0$ on $w$. Since $w$ generates $V(\lambda)$ we see that $\sum_{i=1}^{n+1}x_if_i=0$ and so $\dim\End(V(\lambda))\leq n$. 

Fix an isomorphism $V(\lambda)\cong{_1K^n_\phi}$, and let $\{e_1,\dots,e_n\}$ be the standard basis for $K^n$. Given $f\in\End(V(\lambda))$, we can write $f(e_i)=\sum_jf_{ji}e_j$, where $f_{ji}\in K$.  Then the map $f\mapsto M(f)=(f_{ij})$ allows us to realize each $f\in\End(V(\lambda))$ as right multiplication by the matrix $M(f)\in M_n(K)$.  The fact that $f$ is a bimodule endomorphism is equivalent to $M(f)$ commuting with $\phi(x)$ for all $x\in K$.  Conversely, if $M \in M_n(K)$ with $M=(m_{ij})$ and if $M$ commutes with $\phi(x)$, the rule $e_{i} \mapsto \sum_jm_{ji}e_j$ makes $M$ an element of $\End(V(\lambda))$.  

For each $p\leq n$, let $M(p)$ be the matrix given by $M(p)_{ij}=\b_{pji}$.  We prove that $M(p) \in \End(V(\lambda))$.  If $v=(\a_1,\dots,\a_n)$ as in \eqref{eqn.v}, then one calculates that 
\begin{equation}
\sigma(v)M(p)=(\sigma(\a_1),\dots,\sigma(\a_n))(\beta_{pji})=\bigl(\sum_{j}\beta_{p1j}\sigma(\a_j),\dots,\sum_j\beta_{pnj}\sigma(\a_j)\bigr)\end{equation}
for all $\sigma\in\Aut(\bar K/K)$. On the other hand, 
\begin{equation}
\sigma(\a_p)\sigma(\a_i)=\sigma(\a_p\a_i)=\sigma(\sum_j\beta_{pij}\a_j)=\sum_j\beta_{pij}\sigma(\a_j).\end{equation}
Hence we see that the $i$-th component of $\sigma(v)M(p)$ is $\sigma(\a_p)\sigma(\a_i)$, and we conclude that $\sigma(v)$ is an eigenvector for $M(p)$ with eigenvalue $\sigma(\a_p)$; in particular, we see that $\sigma(v)M(p)M(q)=\sigma(v)M(q)M(p)$ for all $p,q\leq n$ and $\sigma\in\Aut(\bar K/K)$.  Since $\{v,\sigma_2(v),\dots,\sigma_n(v)\}$ is a basis for $\bar K^n$, we conclude that in fact $M(p)$ and $M(q)$ commute for all $p,q$.  Finally, since $\sigma_i(v)$ is a common eigenvector for $\phi(x)$ and $M(p)$ for all $p\leq n$ and $x\in K$, we see that $M(p)$ and $\phi(x)$ commute.  Therefore, $\{M(1),\dots, M(n)\}$ are pairwise commuting, $K$-linearly independent elements of $\End(V(\lambda))$.  Since $\dim\End(V(\lambda))\leq n$, we conclude that $\End(V(\lambda))\cong K\{M(1),\dots, M(n)\}$; one checks easily that the map $M(p)\mapsto\a_p$ gives the desired ring isomorphism $\End(V(\lambda))\cong K(\lambda)$.
\end{proof}

We illustrate Theorem \ref{simple theorem} with several examples.

\begin{example} Suppose that there exists $\lambda\in\Emb(K)$ with $|\lambda^G|=2$.  Then $K(\lambda)$ is a degree 2 extension of $K$, and so $K(\lambda)=K(\sqrt{m})$ for some $m\in K$. Then $\Aut(K(\sqrt{m})/K)$ is generated by $\sigma$, where $\sigma(\sqrt{m})=-\sqrt{m}$, and $\lambda^G=\{\lambda,\sigma\lambda\}$.  

Using $\{1,\sqrt{m}\}$ as a $K$-basis for $\sqrt{m}$, we can write $\lambda(x)=\lambda_1(x)+\lambda_2(x)\sqrt{m}$.  If we write out the matrix $(\phi_{ij}(x))$, we see that 
\begin{equation} \label{eqn.mat}
\phi(x)=\begin{pmatrix} \lambda_1(x) &m\lambda_2(x)\\ \lambda_2(x)&\lambda_1(x)\end{pmatrix},
\end{equation}
and that 
\[(1,\sqrt{m})\begin{pmatrix} \lambda_1(x) &m\lambda_2(x)\\ \lambda_2(x)&\lambda_1(x)\end{pmatrix}=(\lambda(x),\lambda(x)\sqrt{m})=\lambda(x)(1,\sqrt{m}).\]
Moreover,  The fact that $\phi$ is a homomorphism gives the formulas 
\[\begin{split}
\lambda_1(xy)&=\lambda_1(x)\lambda_1(y)+m\lambda_2(x)\lambda_2(y)\\
\lambda_2(xy)&=\lambda_1(x)\lambda_2(y)+\lambda_1(y)\lambda_2(x).
\end{split}\]
Thus we recover \cite[Theorem 1.3(iii)]{pat1} as a special case of Theorem \ref{simple theorem}. \qed
\end{example}

\begin{example}
Let $K$ be a field of characteristic different from $3$, and suppose there exists $\lambda \in \operatorname{Emb}(K)$ such that $|\lambda^G|=3$.  Then $[K(\lambda):K]=3$, so that $K(\lambda)=K(\gamma)$, where $\gamma$ is the root of an irreducible polynomial $x^{3}+bx+c$ with $b,c \in K$.  Thus, $\{1, \gamma, \gamma^2\}$ is a basis of $K(\lambda)/K$.  In this basis, we find that $(\phi_{ij}(x))$ is the matrix
\[\begin{pmatrix}
\lambda_{0}(x) & -c\lambda_{2}(x) & -c\lambda_{1}(x) \\
\lambda_{1}(x) & \lambda_{0}(x)-b\lambda_{2}(x) & -b\lambda_{1}(x)-c\lambda_{2}(x) \\
\lambda_{2}(x) & \lambda_{1}(x) & \lambda_{0}(x) - b\lambda_{2}(x)
\end{pmatrix}\]  
where the functions $\lambda_{0}$, $\lambda_{1}$ and $\lambda_{2}$ satisfy the relations

\begin{equation*}
\begin{split}
\lambda_{0}(xy)& =\lambda_{0}(x)\lambda_{0}(y)-c\lambda_{2}(x)\lambda_{1}(y)-c\lambda_{1}(x)\lambda_{2}(y)\\
\lambda_{1}(xy)& =\lambda_{1}(x)\lambda_{0}(y)+(\lambda_{0}(x)+b\lambda_{2}(x))\lambda_{1}(y)-(b\lambda_{1}(x)+c\lambda_{2}(x))\lambda_{2}(y))\\
\lambda_{2}(xy)& =\lambda_{2}(x)\lambda_{0}(y)+\lambda_{1}(x)\lambda_{1}(y)+(\lambda_{0}(x)-b\lambda_{2}(x))\lambda_2(y).
\end{split}
\end{equation*}
\qed

\end{example}

\begin{example} \label{example.rho}
Suppose $p\geq 3$ is prime, $\rho =\sqrt[p]{2}$, $\zeta$ is a primitive $p$-th root of unity, $k=\Q$ and $K=\Q(\rho)$.    Then $K(\zeta)$ is the Galois closure of $K/\Q$, with $\Aut(K(\zeta)/K)=\{\sigma_i:1\leq i\leq p-1\}$, where $\sigma_i(\zeta)=\zeta^i$.  If we let $\lambda:K\rightarrow \bar K$ be the embedding that takes $\rho$ to $\zeta\rho$, then $\Emb(K)=\{\Id_K\}\cup\{\sigma_i\lambda:1\leq i\leq p-1\}$. Hence $\Lambda(K)$ consists of the two orbits $\{\Id_K\}$ and $\lambda^G=\{\sigma_i\lambda:1\leq i\leq p-1\}$, and so there are up to isomorphism two simples in $\Vect(K)$: the trivial simple bimodule $K$ corresponding to $\{\Id_K\}$, and a $p-1$-dimensional simple corresponding to $\lambda^G$.   

We now construct the matrix homomorphism $\phi:K\rightarrow M_{p-1}(K)$ representing the $p-1$-dimensional simple as in \eqref{phi equation}.  First, taking $\{1,\zeta,\dots,\zeta^{p-2}\}$ as a basis of $K(\zeta)/K$ and letting $\alpha_{i}=\zeta^{i}$ for $0 \leq i \leq p-2$ (we have shifted our indices for ease of computation), we compute the constants $\beta_{jki}$:  if $j+k \neq p-1$, then 
\[\alpha_{j}\alpha_{k}=\zeta^{j}\zeta^{k}=\zeta^{j+k}=\alpha_{j+k},\]
where the superscripts and subscripts are taken modulo $p$.  Therefore, when $j+k \neq p-1$, $\beta_{jki}= 1$ if and only if $i \equiv j+k \pmod{p}$, and $\beta_{jki}=0$ otherwise.

If $j+k=p-1$, then 
\[\alpha_{j}\alpha_{k}=\zeta^{p-1}=-1-\zeta- \cdots -\zeta^{p-2} = -\alpha_{0}-\alpha_{1}- \cdots - \alpha_{p-2}.\]
Therefore, when $j+k=p-1$, $\beta_{jki}=-1$ for all $0 \leq i \leq p-2$.  

Thus, 
\begin{enumerate}
\item if $j=0$ then $j+k \neq p-1$, so $\beta_{0ki}=\delta_{ki}$, and  

\item if $j \neq 0$, either
\begin{enumerate}
\item
$k=p-1-j$, in which case $\beta_{j,p-1-j,i}=-1$ for all $i$ or 

\item $k \neq p-1-j$, in which case $\beta_{j,i-j,i}=1$ for all $i \neq j-1$ (where subscripts are taken modulo $p$) and $\beta_{jki}=0$ otherwise. 
\end{enumerate}
\end{enumerate}
Next, we write
\[\lambda(x)=\lambda_{0}(x)+\lambda_{1}(x)\zeta+\cdots+\lambda_{p-2}(x)\zeta^{p-2}\]
and determine the functions $\lambda_{i}(x)$, $0 \leq i \leq p-2$.  If $x \in K$, we may write $x=\sum_{l=0}^{p-1}a_{l}\rho^{l}$ with $a_{0}, \ldots, a_{p-1} \in \Q$.  It is then easy to see that 
$$
\lambda_{i}(\sum_{l=0}^{p-1}a_{l}\rho^{l})=a_{i}\rho^{i}-a_{p-1}\rho^{p-1}
$$  
for $0 \leq i \leq p-2$.

Using the formula $\phi_{ij}(x)=\sum_{k=0}^{p-2}\beta_{jki}\lambda_{k}(x)$, we may deduce that $\phi(x)$ is the matrix
\begin{equation} \label{eqn.matprime}
\begin{pmatrix}
\lambda_{0}(x) & -\lambda_{1}(x) \\
\lambda_{1}(x) & -\lambda_{1}(x)+\lambda_{0}(x)
\end{pmatrix}
\end{equation}
when $p=3$ and $\phi(x)$ is the matrix
$$
\begin{pmatrix}
\lambda_{0}(x) & -\lambda_{p-2}(x) & -\lambda_{p-3}(x)+\lambda_{p-2}(x) & \cdots & -\lambda_{1}(x)+\lambda_{2}(x) \\
\lambda_{1}(x) & -\lambda_{p-2}(x)+\lambda_{0}(x) & -\lambda_{p-3}(x) & \cdots & -\lambda_{1}(x)+\lambda_{3}(x) \\
\lambda_{2}(x) & -\lambda_{p-2}(x)+\lambda_{1}(x) & -\lambda_{p-3}(x)+\lambda_{0}(x) & \cdots & -\lambda_{1}(x)+\lambda_{4}(x) \\
\vdots         &        \vdots                    &    \vdots         &        & \vdots \\
\lambda_{p-3}(x) & -\lambda_{p-2}(x)+\lambda_{p-4}(x) & -\lambda_{p-3}(x)+\lambda_{p-5}(x) & \cdots & -\lambda_{1}(x) \\
\lambda_{p-2}(x) & -\lambda_{p-2}(x)+\lambda_{p-3}(x) & -\lambda_{p-3}+\lambda_{p-4}(x) & \cdots & -\lambda_{1}(x)+\lambda_{0}(x)
\end{pmatrix}
$$
when $p \geq 5$.  

Had we chosen a different basis for $K(\lambda)$ over $K$, then $\phi(x)$ would have a different form.  For example, when $p=3$, then $K(\zeta)=K(\sqrt{-3})$.  If we use $\{1,\sqrt{-3}\}$ as a basis for $K(\zeta)$ over $K$, then we find that $\phi(x)$ takes the form \eqref{eqn.mat}. \qed
\end{example}

We conclude this section by noting that there are no nontrivial extensions between nonisomorphic simple bimodules.  The result is probably well known, but we were unable to find a reference.

\begin{prop} $\Ext^1_{K^e}(V,W)=0$ for nonisomorphic simple bimodules $V,W$. \label{ext prop}
\end{prop}

\begin{proof} Suppose that $V$ and $W$ are nonisomorphic simple bimodules.  Let $\Phi([V])=\lambda^G$ and $\Phi([W])=\mu^G$, and fix isomorphisms $V\cong {_1K^m_\phi}$ and $W\cong {_1K^n_\psi}$.  If $U$ is an extension of $V$ by $W$, then there is a basis for $K^{m+n}$ such that $U\cong {_1K^{m+n}_\eta}$, where $\eta=\begin{pmatrix} \psi& \theta\\0&\phi\end{pmatrix}$ for some $\theta:K\rightarrow M_{n\times m}(K)$.  Enumerate the elements of $\lambda^G$ and $\mu^G$ as $\{\lambda_1,\dots,\lambda_m\}$ and $\{\mu_1,\dots,\mu_n\}$, respectively.  Then we have that 
\begin{equation} \begin{pmatrix}\phi& \theta\\0&\psi\end{pmatrix}\sim_{\bar K} \diag(\lambda_1,\dots,\lambda_m,\mu_1\dots,\mu_n)\sim_{\bar K}
\begin{pmatrix}\phi&0\\ 0&\psi\end{pmatrix}.
\end{equation}
It follows by Lemma \ref{Piziak lemma} that $U\cong V\oplus W$.
\end{proof} 
\section{Algebraic $K$-theory of $\Vect(K)$}

We shall denote by $K^B_i(K)$ the Quillen $K$-theory of $\Vect(K)$ \cite{Quillen} (the superscript stands for ``bimodule").  The description of the simples in $\Vect(K)$ in Section 3 and the Devissage Theorem \cite[Corollary 5.1]{Quillen} immediately yield the following result.

\begin{thm} For all $i\geq 0$, there is an isomorphism of abelian groups 
\begin{equation} K_i^B(K)\cong\bigoplus_{\lambda^G\in\Lambda(K)}K_i(K(\lambda)).\end{equation} \label{kgroup theorem}
\end{thm}

The Grothendieck group $K_0^B(K)$ can be made into a commutative ring by defining multiplication via the tensor product.  Thus, if $V$ and $W$ are simple bimodules in $\Vect(K)$, we define $[V]\cdot[W]=[V\otimes W]$ in $K_0^B(K)$.  (Here and below $\otimes$ denotes the tensor product over $K$.)  In particular, $[V]\cdot[W]=\sum_{i=1}^t[V_i]$, where $V_1,\dots, V_t$ are the composition factors of $V\otimes W$.  There is an especially nice description of $K_0^B(K)$ when $\Emb(K)=\Aut(K)$; this will happen for instance if $K$ is a normal algebraic extension of the centralizing subfield $k$.

\begin{prop} If $K$ is a field with $\Emb(K)=\Aut(K)$, then there is a ring isomorphism $K_0^B(K)\cong\Z[\Aut(K)]$.\label{normal field prop}
\end{prop}

\begin{proof}
Each simple bimodule in $\Vect(K)$ is isomorphic to $_1K_\phi$ for some $\phi\in \Aut(K)$.  The map $[_1K_\phi]\mapsto\phi$ then gives an isomorphism between the abelian groups $K_0^B(K)$ and $\Z[\Aut(K)]$.  Moreover, an elementary calculation shows that $_1K_\phi\otimes{_1K_\psi}\cong{_1K_{\phi\psi}}$.  From this it follows readily that the above map is actually a ring isomorphism.  
\end{proof}

In order to describe the ring structure of $K_0^B(K)$ for a general field $K$, we shall need to introduce some  notation.  Identifying the $K$-algebras $M_m(K)\otimes M_n(K)$ and $M_{mn}(K)$, we introduce multi-index notation to refer to the coordinates of $M_{mn}(K)$ as follows.  Order the pairs $(i,j)$ with $1\leq i\leq m$ and $1\leq j\leq n$ lexicographically; then there is a bijection between these pairs and $\{1,\dots, mn\}$.  We shall write $A_{(i_1,i_2),(j_1,j_2)}$ for the entry of $A\in M_{mn}(K)$ whose row corresponds to $(i_1,i_2)$ and whose column corresponds to $(j_1,j_2)$ under this bijection.  The reason for adopting this notation is that, if $A=(a_{ij})\in M_m(K)$ and $B=(b_{ij})\in M_n(K)$, then $(A\otimes B)_{(i_1,i_2),(j_1,j_2)}=a_{i_1j_1}b_{i_2j_2}$, where $A\otimes B$ is the Kronecker product of $A$ and $B$.

The following is a variant of the Kronecker product for functions.

\begin{defn} Let $\phi=(\phi_{ij}):K\rightarrow M_m(K)$ and $\psi=(\psi_{ij}):K\rightarrow M_n(K)$ be functions.  Then we define their \emph{Kronecker composition} $\phi\otimes\psi:K\rightarrow M_{mn}(K)$ by the rule $(\phi\otimes\psi)_{(i_1,i_2),(j_1,j_2)}=\phi_{i_1j_1}\circ\psi_{i_2j_2}$.  Similarly, if $A=(a_{ij})\in M_n(K)$, then we define $\phi\otimes A\in M_{mn}(K)$ to be the matrix given by $(\phi\otimes A)_{(i_1,i_2),(j_1,j_2)}=\phi_{i_1j_1}(a_{i_2j_2})$.  Note that if $x\in K$, then $(\phi\otimes\psi)(x)=\phi\otimes(\psi(x))$, so that the two definitions are consistent with each other.  Finally, if $B\in M_m(K)$, then we define the functions $\phi B$ and $B\phi$ by $(\phi B)(x)=\phi(x)B$ and $(B\phi)(x)=B\phi(x)$, respectively.
\end{defn}

The utility of the Kronecker composition in understanding tensor products of bimodules is revealed in the following lemma.  In particular, it implies that when $\phi:K\rightarrow M_m(K)$ and $\psi:K\rightarrow M_n(K)$ are homomorphisms, so too is $\phi\otimes\psi:K\rightarrow M_{mn}(K)$.

\begin{lemma}  Given homomorphisms $\phi:K\rightarrow M_m(K)$ and $\psi:K\rightarrow M_n(K)$, we have $_1K^m_\phi\otimes{_1K^n_\psi}\cong {_1K^{mn}_{\phi\otimes\psi}}.$
\end{lemma}

\begin{proof}  Let $\{e_1,\dots, e_m\}$ and $\{f_1,\dots, f_n\}$ be the standard left bases for $K^m$ and $K^n$, respectively. If we let $e_{(i,j)}=e_i\otimes f_j$, then $\{e_{(i,j)}:1\leq i\leq m,\ 1\leq j\leq n\}$ gives a left basis for $K^{mn}$.  We compute the right action of $K$ on $K^{mn}$ under this basis:
\begin{equation}
\begin{split}
e_{(i,j)}\cdot x&=(e_i\otimes f_j)\cdot x=e_i\otimes f_j\psi(x)=e_i\otimes\sum_{l=1}^n\psi_{jl}(x)f_l\\
&=\sum_{l=1}^ne_i\cdot\psi_{jl}(x)\otimes f_l=\sum_{l=1}^ne_i\phi(\psi_{jl}(x))\otimes f_l\\&=\sum_{l=1}^n\sum_{k=1}^m\phi_{ik}(\psi_{jl}(x))e_k\otimes f_l=\sum_{(k,l)}\phi_{ik}\circ\psi_{jl}(x)e_{(k,l)}.\\
&=\sum_{(k,l)}(\phi\otimes\psi)_{(i,j),(k,l)}(x)e_{(k,l)}=e_{(i,j)}(\phi\otimes\psi)(x).
\end{split}
\end{equation}
\end{proof}

\begin{lemma} \label{lemma.kron} Let $\phi:K\rightarrow M_m(K)$ and $\psi:K\rightarrow M_n(K)$ be homomorphisms and let $A=(a_{ij})\in M_m(K)$, $B=(b_{ij}),C=(c_{ij})\in M_n(K)$. Then the following hold:
\begin{enumerate}
\item $(\phi\otimes B)(\phi\otimes C)=\phi\otimes BC$.
\item $(A\otimes I_n)(\phi\otimes B)=(A\phi)\otimes B$ and $(\phi\otimes B)(A\otimes I_n)=(\phi A)\otimes B$, where $I_n$ is the $n\times n$ identity matrix.
\item If $\phi\sim\phi'$ and $\psi\sim\psi'$, then $\phi\otimes \psi\sim\phi'\otimes\psi'$.
\end{enumerate}
\end{lemma}

\begin{proof}  (1) We compute the $(i_1,i_2),(j_1,j_2)$ component of $(\phi\otimes B)(\phi\otimes C)$:
\begin{equation}
\begin{split}
(\phi\otimes B)(\phi\otimes C)_{(i_1,i_2),(j_1,j_2)}&=\sum_{(k,l)}(\phi\otimes B)_{(i_1,i_2),(k,l)}(\phi\otimes C)_{(k,l),(j_1,j_2)}\\
&=\sum_{(k,l)}\phi_{i_1k}(b_{i_2l})\phi_{kj_1}(c_{lj_2})\\
&=\sum_l\phi_{i_1j_1}(b_{i_2l}c_{lj_2})\ \ \ \mbox{($\phi$ is a homomorphism)}\\
&=\phi_{i_1j_1}((BC)_{i_2j_2})=(\phi\otimes BC)_{(i_1,i_2),(j_1,j_2)}.
\end{split}
\end{equation}

(2) Again, the proof is a computation.  We show the first equality and leave the second to the reader.
\begin{equation}
\begin{split}
(A\otimes I_n)(\phi\otimes B)_{(i_1,i_2),(j_1,j_2)}&=\sum_{(k,l)}(A\otimes I_n)_{(i_1,i_2),(k,l)}(\phi\otimes B)_{(k,l),(j_1,j_2)}\\
&=\sum_{(k,l)}a_{i_1k}(I_n)_{i_2l}\phi_{kj_1}(b_{lj_2}).
\end{split}
\end{equation}
The only nonzero term in the sum occurs when $l=i_2$, because of the $(I_n)_{i_2l}$ term. Hence the above sum collapses to 
\begin{equation}
\sum_ka_{i_1k}\phi_{kj_1}(b_{i_2j_2})=(A\phi)_{i_1j_1}(b_{i_2j_2})=(A\phi\otimes B)_{(i_1,i_2),(j_1,j_2)}.
\end{equation}

(3)  First, suppose that $B\in M_n(K)$ is invertible.  Then by part (1), we have $(\phi\otimes B)(\phi\otimes B^{-1})=\phi\otimes BB^{-1}=\phi\otimes I_n=I_{mn}$.  Thus $\phi\otimes B$ is invertible, with inverse $\phi\otimes B^{-1}$.  Now, suppose that $B\psi(x)B^{-1}=\psi'(x)$ for all $x\in K$, and that $A\phi(x)A^{-1}=\phi'(x)$ for all $x\in K$.  Then, for all $x\in K$, we have
\begin{equation}
(A\otimes I_n)(\phi\otimes B)(\phi\otimes \psi(x))(\phi\otimes B^{-1})(A^{-1}\otimes I_n)=(\phi'\otimes \psi')(x)
\end{equation}
by parts (1) and (2) above.  Hence $\phi\otimes\psi\sim\phi'\otimes\psi'$.
\end{proof}

Any embedding $\lambda$ of $K$ into $\bar K$ can be lifted to an automorphism $\bar\lambda$ of $\bar K$, such that $\bar\lambda|_K=\lambda$.  The following lemma extends this to certain homomorphisms $\phi:K\rightarrow M_n(K)$.

\begin{lemma} If $\phi:K\rightarrow M_n(K)$ represents a simple bimodule, then there exists a homomorphism $\bar\phi:\bar K\rightarrow M_n(\bar K)$ such that $\bar\phi|_K=\phi$.
\end{lemma}

\begin{proof} Write $_1K^m_\phi\cong V(\lambda)$ for some $\lambda^G\in\Lambda(K)$, and write $\lambda^G=\{\lambda_1,\dots,\lambda_m\}$.  Viewing $\phi$ as a function from $K$ to $M_m(\bar K)$, there exists $P\in GL_m(\bar K)$ such that $\phi(x)=P\diag(\lambda_1(x),\dots,\lambda_m(x))P^{-1}$ for all $x\in K$.  Lift each $\lambda_i$ to $\bar\lambda_i:\bar K\rightarrow \bar K$, and define $\bar\phi$ by the formula
\begin{equation}
\bar\phi(x)=P\diag(\bar\lambda_1(x),\dots,\bar\lambda_m(x))P^{-1}.
\end{equation}
Then one easily checks that $\bar\phi$ is a lift of $\phi$.
\end{proof}

The above result obviously extends to semisimple bimodules by induction, but we will only need to apply it in the case where $V$ is simple.   

\begin{thm} Let $\lambda^G,\mu^G\in \Lambda(K)$.  Then $V(\lambda)\otimes V(\mu)$ is semisimple.\label{semisimple thm}
\end{thm}

\begin{proof} If $K$ is finite, then each of $\lambda$ and $\mu$ is an automorphism of $K$, and $V(\lambda)\otimes V(\mu)\cong V(\lambda\mu)$ is simple.  So we may assume that $K$ is infinite. Enumerate the elements of $\lambda^G$ and $\mu^G$ as $\{\lambda_1,\dots,\lambda_m\}$ and $\{\mu_1,\dots,\mu_n\}$ respectively, and let $\bar\lambda_i$ and $\bar\mu_j$ be lifts of $\lambda_i$ and $\mu_j$ to automorphisms of $\bar K$.  If we write $V(\lambda)\cong{_1K^m_\phi}$ and $V(\mu)\cong{_1K^n_\psi}$, then the previous lemma shows that there are lifts $\bar\phi:\bar K\rightarrow M_m(\bar K)$ and $\bar\psi:\bar K\rightarrow M_n(\bar K)$, such that $\bar\phi\sim\diag(\bar\lambda_1,\dots,\bar\lambda_m)$ and $\bar\psi\sim\diag(\bar\mu_1,\dots,\bar\mu_n)$.  It follows from Lemma \ref{lemma.kron} and an elementary calculation that $\bar\phi\otimes\bar\psi\sim\diag(\bar\lambda_i\bar\mu_j:1\leq i\leq m,\ 1\leq j\leq n)$.  

For each pair $(i,j)$, let $\nu_{ij}=\bar\lambda_i\bar\mu_j|_K$.  Then $\nu_{ij}\in\Emb(K)$ and $\nu_{ij}^G\in\Lambda(K)$.  Moreover, an easy calculation shows that $\bar\phi\otimes\bar\psi|_K=\phi\otimes\psi$, and from this we conclude that $\phi\otimes\psi\sim_{\bar K} \diag(\nu_{ij})$.  Partition the multiset $\{\nu_{ij}\}$ into a union of disjoint orbits, counting multiplicities, say $\{\nu_{ij}\}=\bigcup_{k=1}^t(m_k)\nu_k^G$, where $(m_k)\nu_k^G$ means $m_k$ copies of $\nu_k^G$.  Let $V=\oplus_{k=1}^tV(\nu_k)^{(m_k)}$ and write $V\cong{_1K^{mn}_\theta}$ for some $\theta$. Then by construction we have that $\phi\otimes\psi\sim_{\bar K}\theta$; by Lemma \ref{Piziak lemma} $\phi\otimes\psi\sim\theta$, so that $V(\lambda)\otimes V(\mu)\cong\oplus_{k=1}^t V(\nu_k)^{(m_k)}$ is semisimple.
\end{proof}

The above theorem yields a presentation for the ring $K_0^B(K)$ by generators and relations.  We distinguish between the trivial simple bimodule $K$ which corresponds to $\{\Id_K\}\in\Lambda(K)$ and acts as the identity of $K_0^B(K)$, and the nontrivial simple bimodules $\{V(\lambda):\lambda^G\neq\{\Id_K\}\}$.

\begin{cor} Write $\Lambda(K)=\{\Id_K\}\cup(\bigcup_{i\in I}\lambda_i^G)$ as a union of disjoint orbits, and for each pair $i,j$, write $V(\lambda_i)\otimes V(\lambda_j)\cong K^{(\alpha_{ij})}\oplus(\bigoplus_{l\in I}V(\lambda_l)^{(\alpha_{ijl})})$ for nonnegative integers $\alpha_{ij},\alpha_{ijl}$.  Then $K_0^B(K)$ is isomorphic to the quotient of $\Z\langle\{ x_i:i\in I\}\rangle$ by the ideal $I$ generated by $\{x_ix_j-(\sum_{l\in I}\alpha_{ijl}x_l+\alpha_{ij}):i,j\in I\}$.
\end{cor}

The following example illustrates how one can use Theorem \ref{semisimple thm} to find an explicit presentation for $K_0^B(K)$.

\begin{example} 
Let $p$ be an odd prime and let $K=\Q(\rho)$, where $\rho$ is a real $p$-th root of 2.  As in Example \ref{example.rho}, $\Emb(K)$ is partitioned into two orbits: $\Emb(K)=\{\Id_K\}\cup\lambda^G$, where $\lambda$ is the embedding defined by $\lambda(\rho)=\zeta\rho$. Now, $\Aut(K(\zeta)/K)$ is cyclic of order $p-1$; let $\sigma$ be a generator for $\Aut(K(\zeta)/K)$.  To be precise, let $\sigma(\zeta)=\zeta^n$, where $n$ is a multiplicative generator for $(\Z/p\Z)^*$.  There are obvious lifts of $\lambda$ and $\sigma$ to automorphisms of $\bar K$; we abuse notation and denote these lifts by $\lambda$ and $\sigma$ as well.

There are exactly two simple bimodules up to isomorphism:  The trivial bimodule $K$, and the $p-1$-dimensional bimodule $V(\lambda)$.  In order to calculate the ring structure on $K_0^B(K)$, we must decompose $V(\lambda)\otimes V(\lambda)$ as a direct sum of simples.  

If we write $V(\lambda)\cong{_1K^{p-1}_\phi}$, then $\phi\sim_{\bar K}\diag(\sigma^i\lambda:0\leq i\leq p-2)$.  Hence $\phi\otimes\phi\sim_{\bar K} \diag(\sigma^i\lambda\sigma^j\lambda:0\leq i,j\leq p-2)$.  So, we must count the number of times that $\sigma^i\lambda\sigma^j\lambda|_K=\Id_K$.  

We compute:
\[\sigma^i\lambda\sigma^j\lambda(\rho)=\sigma^i\lambda\sigma^j(\zeta\rho)=\sigma^i\lambda(\zeta^{n^j}\rho)=\sigma^i(\zeta^{n^j+1}\rho)=\zeta^{n^i(n^j+1)}\rho.\]

So, we must have $n^i(n^j+1)\equiv 0\pmod{p}$.  Since $(n,p)=1$, this only happens when $n^j+1\equiv 0\pmod{p}$, and since $n$ is a multiplicative generator for $(\Z/p\Z)^*$, this only happens for $j=(p-1)/2$.  For this value of $j$, we see that $\sigma^i\lambda\sigma^{(p-1)/2}\lambda|_K=\Id_K$ for \emph{all} $i$; in particular, there are exactly  $p-1$ copies of the trivial bimodule as a summand of $V(\lambda)\otimes V(\lambda)$.  

The rest is a dimension count:  Since $\dim V(\lambda)\otimes V(\lambda)=(p-1)^2$ and $V(\lambda)\otimes V(\lambda)\cong K^{(p-1)}\oplus V(\lambda)^{(t)}$, it follows that $t=p-2$; i.e. $V(\lambda)\otimes V(\lambda)\cong K^{(p-1)}\oplus V(\lambda)^{(p-2)}$.

From this we conclude that $K_0^B(K)\cong \Z[x]/(x^2-(p-2)x-(p-1))$.\qed
\end{example}

We conclude this section with a brief discussion of an alternative, ``naive" approach to the Grothendieck ring of $\Vect(K)$.  Namely, one could consider the free abelian group on isomorphism classes in $\Vect(K)$, modulo only those relations induced by direct sums (instead of all exact sequences).  We denote this ring by $K_0^\oplus(K)$.  Our aim is to show that, while $K_0^B(K)$ is computable in many cases, $K_0^\oplus(K)$ is an intractable object of study.  We begin with a definition.

\begin{defn}  A \emph{higher $k$-derivation of order $m$} (or an \emph{$m$-derivation}) on $K$ is a sequence of $k$-linear maps $\textbf{d}=\{d_0,d_1,\dots, d_m\}$, such that $d_l(xy)=\sum_{i+j=l}d_i(x)d_j(y)$ for all $x,y\in K$.  (In particular $d_0:K\rightarrow K$ is an endomorphism and $d_1$ is a $d_0$-derivation.)  We denote the set of all $n$-derivations by $HS_n(K)$, and the set of all higher derivations (of all orders) by $HS(K)$.  We refer the reader to  \cite[Section 27]{Matsumura} for more information on higher derivations.  (Note that our definition is slightly more general, in that \cite{Matsumura} assumes that $d_0=\Id_K$).
\end{defn}
 
Note that $HS(K)$ can be made into an abelian semigroup with identity as follows:  Given $\textbf{d}=\{d_0,\dots,d_m\}$ and $\textbf{d$'$}=\{d_0',\dots, d_n'\}$, we define $\textbf{d$\cdot$d$'$}=\{\delta_0,\dots, \delta_{m+n}\}$, where $\delta_l=\sum_{i+j=l}d_id_j'$.  (Here we set $d_i=0$ for $i>m$ and $d_j'=0$ for $j>n$.)  The above operation actually makes $HS(K)$ a group, but we will not need this fact below.

Given $\textbf{d}=\{d_0,d_1,\dots, d_m\}$, we define a map $\phi(\textbf{d}):K\rightarrow M_{m+1}(K)$ by
\begin{equation}
\phi(\textbf{d})(x)=\begin{pmatrix}d_0(x)&d_1(x)&\dots&d_m(x)\\ 0&d_0(x)&\ddots&\vdots\\ \vdots&\ddots&\ddots&d_1(x)\\ 0&\dots&0&d_0(x)\end{pmatrix}.
\end{equation}
That is, $\phi(\textbf{d})(x)$ is an upper triangular Toepliz matrix, whose entry on the $i$-th superdiagonal is $d_i(x)$.  The fact that $\textbf{d}\in HS_m(K)$ is precisely the condition that $\phi(\textbf{d})$ is a homomorphism.  It is fairly easy to see that the two-sided vector space $V(\textbf{d})={_1K^{m+1}_{\phi(\textbf{d})}}$ is indecomposable in $\Vect(K)$.  Conversely, if $\phi:K\rightarrow M_{m+1}(K)$ is a homomorphism such that $\phi(x)$ is an upper triangular Toepliz matrix for all $x\in K$, then $\textbf{d}=\{d_0,\dots, d_m\}\in HS_m(K)$, where $d_i(x)$ is the $i$-th superdiagonal of $\phi(x)$.

It follows readily that there is an abelian semigroup homomorphism $\Psi:\Z[HS(K)]\rightarrow K_0^\oplus(K)$.  However, $\Psi$ is in general neither injective nor surjective, and is also not a ring homomorphism.

For instance, let $\textbf{d}=\{d_0,d_1\}$ and let $\textbf{d$'$}=\{d_0,xd_1\}$ for some $x\in K^*$.  Then conjugating $\phi(\textbf{d})$ by $\diag(x,1)$ shows that $V(\textbf{d})\cong V(\textbf{d$'$})$ and so $\Psi$ is not injective.  Similarly, let $V={_1K^3_\phi}$, where 
\[\phi(x)=\begin{pmatrix} d_0(x)&d_1(x)&d_1(x)\\0&d_0(x)&0\\0&0&d_0(x)\end{pmatrix}.\]
Then $V$ is indecomposable but is not in the image of $\Psi$, so $\Psi$ is not surjective.

The fact that $\Psi$ is not a ring homomorphism is easy:  If $\textbf{d}\in HS_m(K)$ and $\textbf{d$'$}\in HS_n(K)$, then $\textbf{d$\cdot$d$'$}\in HS_{m+n}(K)$ and so the left dimension of $V(\textbf{d$\cdot$d$'$})$ is $m+n+1$.  On the other hand, the left dimension of $V(\textbf{d})\otimes V(\textbf{d$'$})$ is $(m+1)(n+1)$.

The above remarks show that in general $K_0^\oplus(K)$ is a more intractable object of study than $K_0^B(K)$, and that its structure depends on significantly subtler arithmetic properties of the field $K$.

\section{Representatives for equivalence classes of matrix homomorphisms}
Let $\phi:K\rightarrow M_n(K)$ be a homomorphism.  In this final section, we consider the problem of finding a representative for the $\sim$-equivalence class of $\phi$ that has a particularly ``nice" form.  

For example, suppose that $\phi(y)$ has all of its eigenvalues in $K$ for some $y\in K$. Then there exists $P\in GL_n(K)$ such that $P\phi(y)P^{-1}$ is in Jordan canonical form.  Let $\lambda_1,\dots, \lambda_t$ be the distinct eigenvalues of $\phi(y)$, with corresponding multiplicities $m_1,\dots, m_t$.  For each $i$, let $n_{i,1},\dots, n_{i,s_i}$ be the sizes of the $\lambda_i$-Jordan blocks of $P\phi(y)P^{-1}$.  Then \cite[Section VIII.2]{gant} implies the following result. (We say that an $m\times n$ matrix $A$ is \emph{generalized upper triangular Toepliz} if it is of the form $\begin{pmatrix} 0& T \end{pmatrix}$ or $\begin{pmatrix} T\\0\end{pmatrix}$, where $T$ is an upper triangular Toepliz matrix.)

\begin{thm}  Assume the above notation.  For all $x\in K$, 
\begin{equation}
P\phi(x)P^{-1}=\diag(\phi_1(x),\dots,\phi_t(x)),\label{Gantmacher equation}
\end{equation}
where each $\phi_i(x)$ is an $m_i\times m_i$-block matrix of the form $\phi_i(x)=(T_{ipq}(x))$, where $T_{ipq}(x)$ is a generalized upper triangular Toepliz matrix of size $n_{i,p}\times n_{i,q}$.\label{Gantmacher thm}
\end{thm}

The above theorem uses nothing more than the description of the set of all matrices which commute with a given matrix in Jordan canonical form; in particular it does \emph{not} use the additional information that $\phi$ is a homomorphism, or that the matrices in $\im\phi$ also commute with each other.  Consequently one can often find a better representation than the one afforded by Theorem \ref{Gantmacher thm}.

\begin{example}  Suppose that $\phi:K\rightarrow M_3(K)$ is such that $\phi(y)=\begin{pmatrix} \lambda&1&0\\0&\lambda&0\\0&0&\lambda\end{pmatrix}$ for some $y\in K$.  Then $\phi(y)$ is in Jordan canonical form, so Theorem \ref{Gantmacher thm}  shows that there exist functions $a,b,c,d,e:K\rightarrow K$ such that 
\begin{equation}
\phi(x)=\begin{pmatrix} a(x)&b(x)&c(x)\\0&a(x)&0\\0&d(x)&e(x)\end{pmatrix}.
\end{equation}
Writing out the condition that $\phi$ is a homomorphism shows that each of $a$ and $e$ are (nonzero) homomorphisms from $K$ to $K$. If we conjugate $\phi(x)$ by the matrix $P=\begin{pmatrix}1&0&0\\0&0&1\\0&1&0\end{pmatrix}$, we see that  $\phi(x)\sim \psi(x)=\begin{pmatrix} a(x)&c(x)&b(x)\\0&e(x)&d(x)\\0&0&a(x)\end{pmatrix}$.  If we let $V={_1K^3_\psi}$, then the composition factors of $V$ are $\{{_1K_a},{_1K_e},{_1K_a}\}$.

Suppose first that $a=e$. Then the fact that $\psi$ is a homomorphism implies that 
$b(x_1x_2)=a(x_1)b(x_2)+c(x_1)d(x_2)+b(x_1)a(x_2)$ for all $x_1,x_2\in K$.  Since $b(x_1x_2)=b(x_2x_1)$ we can equate terms and get that $c(x_1)d(x_2)=c(x_2)d(x_1)$. If $c\neq 0$, then choosing $x_2$ so that $c(x_2)\neq 0$, we see that $d(x)=\alpha c(x)$, where $\alpha=d(x_2)/c(x_2)$.  If $\a\neq 0$, then we can conjugate $\psi$ by $Q=\diag(1,1,\a)$ to conclude that $\phi\sim\begin{pmatrix} a&c&\frac{1}{\a}b\\0&a&c\\0&0&a\end{pmatrix}$. If $\a=0$, then $d=0$ and so $\phi\sim\begin{pmatrix}a&c&b\\0&a&0\\0&0&a\end{pmatrix}$.  Finally, if $a\neq e$ then the fact that there are no nontrivial extensions between nonisomorphic simples shows that $\phi\sim\begin{pmatrix} a&b&0\\0&a&0\\0&0&e\end{pmatrix}$. Thus we conclude that $\phi$ is equivalent to a homomorphism as in Theorem \ref{Gantmacher thm} that is also upper triangular.\qed
\end{example}

Motivated by the above example, we may ask whether a homomorphism $\phi$ is always equivalent to an upper triangular homomorphism or, ideally, an upper triangular homomorphism of the form (\ref{Gantmacher equation}).  Assuming that the matrices in $\im\phi$ have their eigenvalues in $K$, the answer to the first question is ``yes" \cite[p. 100]{humphreys}.  We shall prove that, under certain additional assumptions, the matrices in $\im\phi$ have upper triangular Toepliz diagonals.  We then derive a sufficient condition for an affirmative answer to the second question.  We begin with some elementary reductions.

Given $V\in\Vect(K)$, let $S_1,\dots, S_t$ be a complete list of the pairwise nonisomorphic composition factors of $V$.  Since $\Ext^1(S_i,S_j)=0$ for $i\neq j$, we see that $V\cong V_1\oplus\dots\oplus V_t$, where each $V_i$ has each of its composition factors isomorphic to $S_i$.  Now, if $\phi$ represents $V$ and $\phi_i$ represents $V_i$ for each $i$, then it is clear that $\phi\sim\diag(\phi_1,\dots,\phi_t)$. Thus it suffices to consider the case where the composition factors of $_1K^n_\phi$ are pairwise isomorphic.  We shall further assume that the simple composition factor of $_1K^n_\phi$ is isomorphic to $_1K_a$ for some $a:K\rightarrow K$; we shall say that $\phi$ is \emph{$a$-homogeneous} in this case.  

\begin{lemma}  If $\phi:K\rightarrow M_n(K)$ is $a$-homogenous for some $a:K\rightarrow K$, then $\phi$ is equivalent to an upper triangular homomorphism with each diagonal entry equal to $a$.
\end{lemma}

\begin{proof} We proceed by induction on $n$, the case $n=1$ being trivial.  Let $V={_1K^n_\phi}$.  Then $_1K_a$ is a sub-bimodule of $V$, generated as a left subspace by a single vector $v$.  Choose a basis for $V$ containing $v$ and order it so that $v$ occurs last; then we see that, in this basis, $V\cong{_1K^n_{\tilde\phi}}$, where $\tilde\phi\sim\begin{pmatrix}\psi&\theta\\0&a\end{pmatrix}$ for some $\psi:K\rightarrow M_{n-1}(K)$.  Now, $_1K^{n-1}_\psi$ is also $a$-homogeneous and so by induction is equivalent to an upper triangular homomorphism with each diagonal entry equal to $a$.  The result follows.
\end{proof}

\begin{thm} Let $\phi:K\rightarrow M_n(K)$ be $a$-homogeneous for some $a:K\rightarrow K$.  Then there exist higher derivations ${\rm\bf d}_1,\dots,{\rm\bf d}_t$, each of whose $0$-th components is equal to $a$, such that 
\begin{equation}
\phi\sim \begin{pmatrix} A_{11} &A_{12}&\dots &A_{1t}\\ 0&A_{22}&\dots &A_{2t}\\ \vdots&\ddots &\ddots&\vdots\\0&\dots&0&A_{tt}\end{pmatrix}\label{upper triangular equation}
\end{equation}
where $A_{ii}(x)=\phi({\rm\bf d}_i)(x)$ and $A_{ij}(xy)=\sum_{l=1}^t A_{il}(x)A_{lj}(y)$ for all $x,y\in K$.\label{upper triangular thm}
\end{thm}

\begin{proof} The fact that $A_{ij}(xy)=\sum_lA_{il}(x)A_{lj}(y)$ follows because $\phi$ is a homomorphism; the key is to show that the diagonal matrices $A_{ii}$ have the stated form. By the previous lemma, we may assume without loss of generality that $\phi$ is upper triangular. Write $\phi=(\phi_{ij})$, where $\phi_{ii}=a$ for all $i$ and $\phi_{ij}=0$ for $i>j$.  Let $i_1\leq \dots\leq i_q$ be all of the indices for which $\phi_{i_k,i_k+1}=0$.  Then we can partition $\phi$ into blocks of size $i_1,i_2-i_1,\dots, i_q-i_{q-1},n-i_q$.  If we let $\phi_l$ denote the $l$-th diagonal block in this partition, then each $\phi_l$ has the properties that each of its diagonal entries is equal to $a$, and none of its first superdiagonal entries is identically $0$.

Replacing $\phi$ by $\phi_l$ we may assume without loss of generality that $\phi_{i,i+1}$ is not identically $0$ for any $i$.  After these reductions, we see that the theorem is trivially true when $n=1$ or $2$, so we assume without loss of generality that $n\geq 3$.  If we expand out $\phi_{i,i+2}(xy)$ using the fact that $\phi$ is a homomorphism and $\phi_{ij}=0$ for $i>j$, we obtain 
\begin{equation}\phi_{i,i+2}(xy)=\phi_{ii}(x)\phi_{i,i+2}(y)+\phi_{i,i+1}(x)\phi_{i+1,i+2}(y)+\phi_{i,i+2}(x)\phi_{i+2,i+2}(y)\end{equation}
and a similar equation for $\phi_{i,i+2}(yx)$.  Substituting $\phi_{ii}=\phi_{i+2,i+2}$ and using the fact that $\phi(xy)=\phi(yx)$, we can simplify the resulting equations to obtain
\begin{equation}\phi_{i,i+1}(x)\phi_{i+1,i+2}(y)=\phi_{i,i+1}(y)\phi_{i+1,i+2}(x)\end{equation}
for all $x,y\in K$.  If we choose $y$ such that $\phi_{i,i+1}(y)\neq 0$, then we have that $\phi_{i+1,i+2}(x)=\alpha_i\phi_{i,i+1}(x)$ for all $x\in K$, where $\alpha_i=\phi_{i+1,i+2}(y)/\phi_{i,i+1}(y)$. Note also that $\alpha_i\neq 0$ for any $i$ since we know that $\phi_{i+1,i+2}$ is not identically $0$.

Let $b=\phi_{12}$ and let $\beta_i=\prod_{j\leq i}\a_j$, so that 
\[\phi=\begin{pmatrix} a&b&&&*\\0&a&\beta_1b&&\\\vdots&0&a&\ddots &\\ \vdots& &\ddots &\ddots&\beta_{n-2}b\\0& \dots& \dots& 0&a\end{pmatrix}.\]
Choose $y\in K$ with $b(y)\neq 0$.  An elementary calculation shows that $(\phi(y)-a(y)I_n)^{n-1}$ is the matrix whose only nonzero entry is $\beta_1\dots\beta_{n-2}b(y)^{n-1}$ in its $(1,n)$-position.  This shows that the minimal polynomial for $\phi(y)$ is $(X-a(y))^n$, so that the Jordan canonical form for $\phi(y)$ is a single block of size $n$. If $P\in GL_n(K)$ is such that $P\phi(y)P^{-1}$ is in Jordan canonical form, then Theorem \ref{Gantmacher thm} shows that $P\phi(x) P^{-1}$ is an upper triangular Toepliz matrix with diagonal equal to $a(x)$ for all $x\in K$.  Thus there exists a higher derivation $\textbf{d}$ such that $P\phi P^{-1}=\phi(\textbf{d})$.  This shows that $\phi$ is equivalent to a matrix of the form \eqref{upper triangular equation}. 
\end{proof}

One may ask under what circumstances it is possible to obtain the best of both worlds:  That is, when can we conclude that $\phi$ is equivalent to an upper triangular representation as in \eqref{upper triangular equation}, and also have each $A_{ij}$ be a generalized upper triangular Toepliz matrix as in Theorem \ref{Gantmacher thm}?  Since the Toepliz condition arises out of commuting with a matrix in Jordan canonical form, the following would be a sufficient condition:

\begin{enumerate}
\item[($*$)] Given a homomorphism $\phi$, there exists $y\in K$ and $P\in GL_n(K)$ such that $P\phi(y)P^{-1}$ is in Jordan canonical form and $P\phi(x)P^{-1}$ is upper triangular for all $x\in K$. 
\end{enumerate}
If $\phi$ is an upper triangular homomorphism, then of course condition $(*)$ is satisfied if there exists $y\in K$ and an upper triangular $P\in GL_n(K)$ such that $P\phi(y)P^{-1}$ is in Jordan canonical form.

Condition $(*)$ is not automatic for a given $y$ and $\phi$.  The following example illustrates that, given $y$, there may be no $P$ such that $P\phi(y)P^{-1}$ is in Jordan canonical form and $P\phi(x)P^{-1}$ is upper triangular.

\begin{example} \label{example.fail}
Let $\textbf{d}=\{d_0,d_1,d_2\}$ be a $2$-derivation, and assume that $d_1 \neq 0$ and that there exists a $y \in K$ such that $d_1(y)=0$, $d_2(y)\neq 0$.  Define $\phi:K \rightarrow M_3(K)$ by
\[\phi(x)= \begin{pmatrix} d_{0}(x) & d_{2}(x) & d_{1}(x) \\  0 & d_{0}(x) & 0 \\ 0 & d_{1}(x)& d_{0}(x)\end{pmatrix}.\]
We claim there does not exist a basis in which $\phi(y)$ has Jordan canonical form and the image of $\phi$ is upper triangular. To establish this claim, we describe every $P\in GL_3(K)$ in which the image of $P\phi P^{-1}$ is upper triangular, and show that $P\phi(y)P^{-1}$ is not in Jordan canonical form for any such $P$.  

Since $d_1\neq 0$, it is not hard to see that the only simultaneous eigenvectors for $\im\phi$ are in $W=\span{(0,1,0)}$.  Similarly, the only simultaneous eigenvectors for $\im\phi$ acting on $K^3/W$ are in $\span\{(0,0,1)+W\}$.  From this we conclude that, if $\mathcal{B}$ is a basis with $\im P\phi P^{-1}$ upper triangular, then
\[\mathcal{B}=\{(0,f_{1},0), (0, f_{2}, f_{3}),(f_{4},f_{5},f_{6}): f_{1}, f_{3}, f_{4} \neq 0\}.\]

For such a basis $\mathcal{B}$, we have
\[P\phi(x)P^{-1} = \begin{pmatrix} d_{0}(x) & \frac{f_{4}}{f_{3}}d_{1}(x)  & (\frac{f_{6}f_{3}-f_{4}f_{2}}{f_{1}f_{3}})d_{1}(x)+\frac{f_{4}}{f_{1}}d_{2}(x) \\
0 & d_{0}(x) & \frac{f_{3}}{f_{1}}d_{1}(x)  \\ 0 & 0 & d_{0}(x) \end{pmatrix}.\]
By construction, $P\phi(y)P^{-1}$ is not in Jordan canonical form.

Note that higher derivations satisfying the given hypotheses do exist.  For example, let $K$ be the quotient field of $k[x,y,z]/(xy-z^2)$, where $k$ is a field of characteristic $2$.  In \cite[Example 1.2 and Theorem 1.5]{traves}, a nontrivial $\textbf{d}\in HS_2(K)$ is constructed such that $d_1(x-z)=0$ and $d_2(x-z)=x$.\qed
\end{example}

\begin{defn}  Let $A$ be an $n \times n$ upper triangular matrix with single eigenvalue $\lambda$, and let the Jordan canonical form of $A$ have block sizes $n_1\geq n_2\geq\dots\geq n_p$.  For each $i\leq n$, let $A_i$ be the matrix consisting of the first $i$ rows and columns of $A$.  We say that $A$ is \emph{Jordan-ordered} if, for all $i\leq n$, the dimension of the eigenspace of $A_i$ is $j$, where $j$ is the smallest integer such that $n_1+\dots+n_j\geq i$.
\end{defn}

\begin{example} Let $A=\begin{pmatrix}\lambda&0&1\\0&\lambda&0\\0&0&\lambda\end{pmatrix}$, so that the Jordan canonical form of $A$ has blocks of size $2$ and $1$.  Then $A$ is not Jordan-ordered, because the dimension of the eigenspace of $A_2=\begin{pmatrix}\lambda&0\\0&\lambda\end{pmatrix}$ is $2$ and not $1$.\qed
\end{example}

It is not hard to see that, if $A$ is in Jordan canonical form, then $A$ is Jordan-ordered if and only if the Jordan blocks of $A$ are arranged in decreasing size.

The following is our main result concerning Jordan-ordered matrices.

\begin{thm}  If $A\in M_n(K)$ is Jordan-ordered, then there exists an upper triangular $P\in GL_n(K)$ such that $PAP^{-1}$ is Jordan-ordered and is in Jordan canonical form.  \label{theorem.jcf}
\end{thm}

We begin with a preliminary lemma.

\begin{lemma} Suppose that $A\in M_{n+1}(K)$ has a single eigenvalue $\lambda$ of multiplicity $n$. If 
\begin{equation}A=\begin{pmatrix}   & & &a_1\\   &B& &\vdots\\  & & &a_n \\  0&\dots &0&\lambda\end{pmatrix}\end{equation}
with $B\in M_n(K)$ in Jordan canonical form, then there exists an upper triangular $P\in GL_n(K)$ such that $PAP^{-1}$ is in Jordan canonical form.\label{jcf lemma}
\end{lemma} 

\begin{proof}  Since $A$ has the single eigenvalue $\lambda$, the Jordan canonical form for $A$ must be
\begin{equation}\begin{pmatrix} & && &0\\& B&& &\vdots\\ &  & & &0\\   & & & &a \\ 0& \dots&\dots &0&\lambda\end{pmatrix},\label{A jcf}\end{equation}
with $a=0$ or $1$.  We give the proof when $a=1$, the case $a=0$ being similar and left to the reader.

Let $E_A$ denote the eigenspace of $A$ and suppose that $\dim E_A=m+1$.  Since $e_{n+1}\in E_A$, we can take a basis for $E_A$ containing it; moreover elementary calculations then allow us to assume that the final entry of all other basis elements is $0$.  Thus $E_A$ has a basis of the form
\[\{(0,\dots,0,1),(c_{11},\dots, c_{1n},0),\dots,(c_{m1},\dots,c_{mn},0)\}.\]
Since the last $m$ of these vectors are eigenvectors for $A$, we see that $(a_1,\dots, a_n)$ must be a solution to the system of equations
\begin{equation}
\begin{split}c_{11}x_1+\dots+c_{1n}x_n&=0\\
\vdots\hspace{.5in}& \\
c_{m1}x_1+\dots +c_{mn}x_n&=0\label{system}
\end{split}
\end{equation}
and that 
\[\{(c_{11},\dots, c_{1n}),\dots,(c_{m1},\dots,c_{mn})\}\]
is a set of $m$ linearly independent eigenvectors of $B$.  Because $a=1$ we see that the dimension of the eigenspace $E_B$ of $B$ is also $m+1$, and we note that $e_n\in E_B$.  Since $B$ is in Jordan canonical form, the matrix
\begin{equation}\begin{pmatrix} c_{11}&\dots&c_{1n}\\ \vdots& &\vdots\\c_{m1}&\dots&c_{mn}\end{pmatrix}\label{c matrix}	
\end{equation}
has $n-m-1$ of its columns equal to $0$, and its final column cannot be equal to $0$ since $e_n$ is an eigenvector for $B$.  Thus \eqref{system} can be viewed as a system of $m$ equations in $m+1$ variables, say $x_{i_1},\dots, x_{i_{m+1}}=x_n$.  Since the rows of \eqref{c matrix} are linearly independent, some subset of $m$ columns of \eqref{c matrix} is linearly independent.  Thus the solution space of \eqref{system} is $1$-dimensional.  On the other hand, since $A$ has the given Jordan canonical form, $(x_{i_1},\dots,x_{i_{m+1}})=(0,0,\dots,1)$ must be a solution to \eqref{system}.  Thus we conclude that $(a_{i_1},\dots,a_{i_{m+1}})=(0,0,\dots, c)$ for some $c\in K$.

Consider the system 
\[(\lambda I_n-B)\begin{pmatrix}y_1\\\vdots\\y_n\end{pmatrix}=\begin{pmatrix}a_1\\\vdots\\a_{n-1}\\ 0\end{pmatrix}.\]
Since $B$ is in Jordan canonical form, the image of left multiplication by $\lambda I_n-B$ has each of its $i_1,\dots,i_{m+1}$-components equal to $0$, and also has dimension $n-m-1$.  Since $a_{i_1}=\dots={a_{i_m}}=0$, we see that there is a solution $y_1=b_1,\dots, y_n=b_n$. Let $\vec{b}$ be the column vector $(b_1,\dots, b_n)^T$; then an elementary calculation shows that, if $P=\begin{pmatrix}I_n&\vec{b}\\0&1\end{pmatrix}\in GL_{n+1}(K)$, then 
\[PAP^{-1}=\begin{pmatrix} & && &0\\& B&& &\vdots\\ &  & & &0\\   & & & &c \\ 0& \dots&\dots &0&\lambda\end{pmatrix}.\]
It follows, since the Jordan canonical form for $A$ is \eqref{A jcf}, that $c\neq 0$.  Conjugating by $\diag(1,\dots,1,1/c)$ finishes the proof.\end{proof}

\begin{proof}[Proof of Theorem \ref{theorem.jcf}] We proceed by induction on $n$, the case $n=1$ being trivial.  Since $A$ is upper triangular, $e_n$ is an eigenvector for $A$.  If $A_{n-1}$ denotes the matrix obtained by deleting the last row and column from $A$, then by induction there exists an upper triangular $Q\in GL_{n-1}(K)$ such that $QA_{n-1}Q^{-1}$ is Jordan-ordered and in Jordan canonical form. Let $R=\begin{pmatrix}Q&0\\0&1\end{pmatrix}\in GL_n(K)$; conjugating $A$ by $R$ then gives
\[RAR^{-1}=\begin{pmatrix}   & & &a_1\\   &B& &\vdots\\  & & &a_n \\  0&\dots &0&\lambda\end{pmatrix},\]
where $B$ is the Jordan ordered, Jordan canonical form for $A_{n-1}$.  

Let the Jordan canonical form for $A$ have blocks of sizes $n_1 \geq \cdots \geq n_p$.  If $n_p=1$, then $B$ has blocks of sizes $n_1,\dots, n_{p-1}$, and the Jordan canonical form for $A$ is $\begin{pmatrix} B&0\\0&\lambda\end{pmatrix}$.  By Lemma \ref{jcf lemma}, there is an upper triangular $T\in GL_n(K)$ with $TRAR^{-1}T^{-1}$ Jordan-ordered and in Jordan canonical form.  Thus the theorem follows with $P=TR$ in this case.

If $n_p>1$, then $B$ has blocks of sizes $n_1,\dots, n_{p-1},n_p-1$ and the block of size $n_p-1$ occurs at the bottom of $B$. Thus the Jordan canonical form for $A$ is 
\[\begin{pmatrix} & && &0\\& B&& &\vdots\\ &  & & &0\\   & & & &1 \\ 0& \dots&\dots &0&\lambda\end{pmatrix},\]
and again letting $T$ be as in Lemma \ref{jcf lemma}, we see that $PAP^{-1}$ is Jordan-ordered and in Jordan canonical form for $P=TR$.  
\end{proof}

Combining Theorems \ref{upper triangular thm} and \ref{theorem.jcf}, we can state a sufficient condition for a homomorphism $\phi:K\rightarrow M_n(K)$ to be equivalent to an upper triangular homomorphism which is generalized upper triangular Toepliz.  We state the result in the case where $\phi$ is $a$-homogenous for some $a:K\rightarrow K$.

\begin{cor}  Let $\phi$ be $a$-homogeneous, and let $\psi\sim \phi$, where $\psi$ is a homomorphism in the form \eqref{upper triangular equation}.  If $\psi(y)$ is Jordan-ordered for some $y\in K$, then 
\begin{equation}
\phi\sim \begin{pmatrix} T_{11}&T_{12}&\dots& T_{1s}\\0&T_{22}&\dots&T_{2s}\\\vdots&\ddots&\ddots&\vdots\\0&\dots&0&T_{ss}\end{pmatrix}
\end{equation}
where each $T_{ij}(x)$ is generalized upper triangular Toepliz.\label{jcf cor}
\end{cor}

In particular there exist higher derivations $\textbf{d}_1,\dots,\textbf{d}_s$ such that $T_{ii}=\phi(\textbf{d}_i)$, although these derivations may be different than those in \eqref{upper triangular equation}.


\begin{thebibliography}{99}

\bibitem{cr} C. W. Curtis and I. Reiner, Methods of Representation Theory with Applications to Finite Groups and Orders, Vol. 1, Wiley Classics Library, Wiley, New York, 1990.

\bibitem{gant} F. R. Gantmacher, The Theory of Matrices, Vol. 1, Chelsea Publishing Company, New York, 1960. 

\bibitem{humphreys} J. E. Humphreys, Linear Algebraic Groups, Graduate Texts in Mathematics {\bf 73}, Springer-Verlag, New York, 1981.

\bibitem{Matsumura} H. Matsumura, Commutative Ring Theory, Cambridge
Studies in Advanced Mathematics \textbf{8}, Cambridge University Press, Cambridge, 1986.

\bibitem{pat1}  D. Patrick, \textit{Noncommutative symmetric algebras of two-sided vector spaces}, J. Algebra {\bf 233} (2000), 16--36.

\bibitem{Quillen} D. Quillen, \textit{Higher algebraic K-theory I},  Algebraic
$K$-Theory, Batelle, 1972 (H. Bass, ed.) Lecture Notes in Mathematics
\textbf{341}, Springer-Verlag, Berlin, 1973, 85--147.

\bibitem{traves} W. N. Traves, \textit{Tight closure and differential simplicity}, J. Algebra {\bf 228} (2000), 457--476.

\bibitem{vdb2}  M. Van den Bergh, \textit{Non-commutative $\mathbb{P}^{1}$-bundles over commutative schemes}, submitted.

\bibitem{vdb1}  M. Van den Bergh, \textit{A translation principle for the four-dimensional Sklyanin algebras}, J. Algebra {\bf 184} (1996), 435--490.

\end{thebibliography}
\end{document}